\documentclass[preprint,12pt]{elsarticle}

\usepackage{amssymb}
\usepackage{amsmath}

\newcommand{\N}{\mathbb{N}} 
\newcommand{\Z}{\mathbb{Z}}
\newcommand{\F}{\mathbb{F}}
\newcommand{\Q}{\mathbb{Q}}

\newcommand{\OO}{\mathcal{O}}

\newcommand{\Li}[2]{\left(\frac{#1}{#2}\right)}
\newcommand{\alphas}{\alpha_0,\ldots,\alpha_{d-1}}
\newcommand{\nice}{\textit{convenient} }

\newcommand{\pp}{\mathfrak{p} }
\newcommand{\raddel}{\rad(\Delta_\mathbf{U})}

\DeclareMathOperator{\rad}{rad}
\DeclareMathOperator{\fail}{Fail}
\DeclareMathOperator{\lcm}{lcm}
\DeclareMathOperator{\adj}{adj}
\DeclareMathOperator{\gal}{Gal}
\DeclareMathOperator{\res}{res}
\DeclareMathOperator{\disc}{disc}

\usepackage{amssymb}
\usepackage{amsmath}
\usepackage{amsthm}
\usepackage{amsfonts}
\usepackage[utf8]{inputenc}
\usepackage{xcolor}
\usepackage{makecell}

\newtheorem{lemma}{Lemma}
\newtheorem{theorem}{Theorem}
\newtheorem{defin}{Definition}
\newtheorem{example}{Example}
\newtheorem{corollary}{Corollary}[theorem]

\begin{document}

\begin{frontmatter}



    \title{On Dold condition and fail factor of linear recurrent sequences}


\author{Mateusz Rajs} 

\affiliation{organization={Jagiellonian University},
            addressline={Łojasiewicza 6}, 
            city={Kraków},
            postcode={30-348}, 
            state={Małopolska},
        country={Poland}}

\begin{abstract}
 A sequence $\mathbf{A}$ is said to be realizable if satisfies so called sign and Dold conditions. We will say that a sequence almost satisfies the Dold condition if there exists a constant $c\in\N_+$ such that $(cA_n)_{n\in\N_+}$ satisfies the Dold condition. In this paper we give a characterisation of sequences defined by linear recursion of any order that almost satisfy the Dold condition. We also give an upper bound on the value of $c$.

\end{abstract}




\end{frontmatter}



\section{Introduction}

We denote by $\Z$, $\N$, $\N_+$, $\Q$ the sets of all integers, non-zero integers, positive integers, and rational numbers, respectively. By $\rad(n)$ for $n\in \N_+$ we mean the radical of $n$, i.e. the greatest square-free divisor of $n$. Moreover, we denote a sequence of integers by a bold letter, e.g. $\mathbf{A} =(A_n)_{n\in\N_+}$.

We say that a sequence $\mathbf A$ of non-negative integers is realizable if there exists a map $T:X\to X$ such that $A_n$ is a number of fixed points of the map $T^n$. There also exists the following equivalent condition for a sequence to be realizable: 

\begin{itemize}
    \item [(r1)] $n \mid \sum_{d\mid n} \mu(d)A_{\frac{n}{d}}$ \quad for $n \in \N_+$, 
    \item [(r2)] $\sum_{d\mid n} \mu(d) A_{\frac{n}{d}} \geqslant 0$ \quad for $n \in \N_+$.
\end{itemize}
The condition (r1) is called the Dold condition and (r2) is called the sign condition.

When well-known combinatorial sequences are examined, one can see that some of them satisfy the Dold condition or are even realizable.
In \cite{combi} Zhang showed (among other classes of sequences) that the Apéry's
numbers are realizable.
Moss in \cite{Phd} (Theroem 5.3.2) showed that the sequence of absolute values of Euler numbers also is realizable.
In \cite{Fibonacci} Moss and Ward proved that the sequence of Fibonacci numbers is not almost realizable and also that the sequence $(5F_{n^2})_{n\in\N}$ of Fibonacci numbers with square indices multiplied by 5 is  realizable . 
Beukers et al. \cite{laurent} examined which coefficients of Laurent series
associated to a multivariate rational function satisfy the Dold condition (which they call Gauss congruences).

In \cite{Stirling}  Miska and Ward defined the repair factor $\fail(\mathbf A)$ as the least positive integer $c$, such that the sequence $(cA_n)_{n\in\N_+}$ is realizable. If such a number does not exist we put $\fail(\mathbf A) = \infty$. If $\fail(\mathbf A) < \infty$, then we call the sequence $\textbf{A}$ almost realizable.Note that $\fail(\mathbf{A})$ is the smallest positive integer $c$ such that $(cA_n)_{n\in\N_+}$ satisfies the Dold condition as multiplication of a sequence by a positive number does not change the validity of the sign condition.

In this paper we consider any sequence $\mathbf{A}$ that is defined by linear recurrence of any order with a particular focus on order $2$. We will determine when such a sequence almost satisfies the Dold condition and prove an upper estimation for the value of $\fail(\mathbf{A})$. Additionally we will show that for some $t$ the sequence $({A}_{n^t})_{n\in\N_+}$ almost satisfies the Dold condition.

\section{Preliminaries and notation }

From now on by $\mathbf U$ we will denote any sequence defined by linear recursion
\begin{align} \label{def_Un} \tag{d1}
    U_n = r_1U_{n-1}+ r_2U_{n-2} + \cdots + r_d U_{n-d} 
    \quad\quad \text{for $n\geqslant d$}
\end{align}
where $r_1,\ldots,r_d\in \Z$ . We will call the number $d$ the order of this sequence. The characteristic polynomial of $\mathbf U$ is defined as 
\begin{align}\label{def_Char} \tag{d2}
    C_{\mathbf U}(x) := x^d - r_1x^{d-1} - \cdots - r_d.
\end{align}
We denote its roots by $\alphas$, so $ C_{\mathbf U}(x) = \prod_{i=0}^{d-1}(x-\alpha_i)$. The discriminant of the characteristic polynomial is
\begin{align}\label{def_Delta} \tag{d3}
    \Delta_{\mathbf U} := \prod_{0\leqslant i < j < d} (\alpha_i-\alpha_j)^2.
\end{align}

The definition of the Dold condition given in (r1) is troublesome. It requires calculation of some nontrivial sums. The lemma below proposes an equivalent definition, which is much more convenient.

\begin{lemma} \label{lem_betterDef}
    A sequence $\mathbf U$ satisfies the Dold condition if and only if  the number $U_{p^en} - U_{p^{e-1}n}$ is divisible by $p^e$ for every prime $p$ and $n, e \in\N_+$ such that $p\nmid n$.
\end{lemma}

\begin{proof}
  
    $(\Leftarrow)$
    The following part of the proof is given in \cite{Fibonacci}. Fix $n\in \N_+$ let $n=p_1^{k_1}\cdots p^m_{k_m}$ be the prime decomposition of $n$. We select one of the primes $p_i$ and set $p:=p_i$, $k:=k_i$ and $n=p^ks$. We have
      
    \begin{equation}
    \begin{split}
        \sum_{d\mid n} \mu \left(\frac{n}{d}\right)U_{d} = 
        \sum_{d\mid p^ks} \mu \left(\frac{p^ks}{d}\right)U_{d} =
        \sum_{d\mid s}\left( \mu\left(\frac{s}{d}\right)U_{dp^k} + \mu\left(\frac{sp}{d}\right)U_{dp^{k-1}} \right) = \label{lem_betterDef_e1}\\
        \sum_{d\mid s} \mu\left(\frac{s}{d}\right) \left(U_{dp^k} -U_{dp^{k-1}} \right) 
    \end{split}
   \end{equation}
         
    From the assumption, we know that $p^k\mid U_{dp^k} -U_{dp^{k-1}}$ hence $p_i^{k_i}= p^k\mid\sum_{d\mid n} \mu \left(\frac{n}{d}\right)U_{d} $. Since the choice of $p_i$ and $k_i$ was arbitrary, it follows that  $n\mid\sum_{d\mid n} \mu \left(\frac{n}{d}\right)U_{d} $.

    $(\Rightarrow)$
    Let $p$ be any prime and $e$ any positive integer. We use induction over $s$ to prove that if $\mathbf U$ satisfies the Dold condition, then $p^e\mid U_{p^es}-U_{p^{e-1}s}$. For $s=1$ the assertion is true. Let us now assume that it is true for $1,2,\ldots,s-1$. Thus,
    \begin{align}
        0 &\equiv 
        \sum_{d\mid p^ks} \mu \left(\frac{p^ks}{d}\right)U_{d} \overset{(\ref{lem_betterDef_e1})}{\equiv}
        \sum_{d\mid s} \mu\left(\frac{s}{d}\right) \left(U_{dp^k} -U_{dp^{k-1}} \right) \equiv \\ &\equiv
        U_{sp^k} -U_{sp^{k-1}} + \sum_{d\mid s,\ d<s} \mu\left(\frac{s}{d}\right) \left(U_{dp^k} -U_{dp^{k-1}} \right) \overset{\text{ind. step}}{\equiv} \\ &\equiv
        U_{sp^k} -U_{sp^{k-1}} + 0 \pmod {p^e}.
    \end{align}

    This ends the proof.
  
\end{proof}

In general, the roots $\alphas$ are not integers. Let us define $K=\Q(\alphas)$ and let $\OO_K$ be the ring of integers of $K$. As $\alphas$ are roots of a monic polynomial they belong to $\OO_K$. Instead of examining equivalences modulo $p$ over $\Z$ we will examine equivalences mod $\pp$ over $\OO_K$, where $\pp$ is a prime ideal of $\OO_K$. 

Let $p$ be any prime in $\Z$. Then, the ideal generated by this prime in $\OO_K$ decomposes into prime ideals over $\OO_K$ as follows:
$$p\OO_K = \prod_i\pp_i^{e_i},$$ where $\pp_i\neq\pp_j$ for $i\neq j$. The exponent $e_i$ is called ramification index and is grater than 1 only for finitely many primes $p$. From now on, $\pp$ will always denote prime ideal of $\OO_K$, $p$ be a unique positive prime integer that belongs to $\pp$, and $e$ will be the ramification index of $p$ in $\OO_K$. 

Now we can write the following equivalent statements:
\begin{enumerate}
    \item[(p1)]  $\mathbf U$ satisfies the Dold condition;
    \item[(p2)] $\forall_{p \text{ prime}}\ \forall_{s,k>0}:\ U_{p^ks} \equiv U_{p^{k-1}s} \pmod {p^k} $;
    \item[(p3)] $\forall_{\pp \text{ prime}}\ \forall_{s,k>0}:\ U_{p^ks} \equiv U_{p^{k-1}s} \pmod {\pp^{ek}} $;
\end{enumerate}

The equivalence (p1) $\Leftrightarrow$ (p2) is exactly the statement of Lemma \ref{lem_betterDef}, while the equivalence (p2) $\Leftrightarrow$ (p3) comes from Chinese remainder theorem.

\vskip 0.3cm
The following lemma will be useful in the sequel.

\begin{lemma}\label{lem_allInOne}
    Let $K$ be a number field and $x,y\in\OO_K$. Let $\pp$ be a prime ideal of $\OO_K$ lying over a prime number $p$. Then 
    $$ x \equiv y \pmod{\pp^{de}} \quad\Rightarrow\quad x^{\pp^k} \equiv y^{p^k} \pmod{\pp^{e(d+k)}}$$
    for $k, d\in \N_+$.
\end{lemma}
\begin{proof}
    We shall establish this proposition through induction. The base case for $k=0$ is given as an assumption. To prove the induction step we assume that $x^{p^{k-1}} \equiv y^{p^{k-1}} \pmod{\pp^{e(d+k-1)}}$. This implies that $x^{p^{k-1}} \equiv y^{p^{k-1}} + z\pmod {\pp^{e(d+k)}}$ for some $z\in \pp^{e(d+k-1)}$. By raising both sides to the power of $p$ we get 
    \begin{align}
        x^{p^k} \equiv 
        \left(y^{p^{k-1}} + z\right)^p \equiv
        y^{{p^{k-1}}\cdot p} + pz(\ldots) + z^p \equiv
        y^{p^{k}} \pmod {\pp^{e(d+k)}}.
    \end{align}
    The last congruence is true because $pz\in \pp^{e(d+k)}$ and $z^p\in\pp^{ep(d+k-1)}  \subset \pp^{e(d+k)}$. (because $d+k \geqslant 2$). This completes the inductive proof.
\end{proof}

\begin{corollary} \label{c_fermat}
    For any $x\in\Z$ by Fermat's little theorem we have $x^p\equiv x\pmod{p}$ so by the above Lemma we get 
    \begin{align}
            x^{p^{k}} \equiv x^{p^{k-1}}\pmod{p^k}
            \quad\quad\text{for any $k\in\N_+$}.
    \end{align}
\end{corollary}

\section{The case of order 2}

This case is a direct generalization of \cite{Fibonacci}.  We will consider only sequences $\mathbf U$ defined by linear recurrence of order 2. This implies that the characteristic polynomial of $ \mathbf U$ $C_{\mathbf U}(x) = x^2-r_1x-r_2$ also has degree 2. 

In this section we will denote the roots of $C_\mathbf{U}$ by $\alpha=\frac{r_1+\sqrt{\Delta}}{2}$ and $\beta=\frac{r_1-\sqrt\Delta}{2}$. We assume for simplicity that the roots are different $\alpha \not= \beta$. By Vieta's formulas we have $r_2 = \alpha\beta$ and $r_1 = \alpha+\beta$. Finally we notice that
\begin{align}
    U_n = l_1\alpha^n + l_2\beta^n \quad \text{for $n\in\N$}, 
\end{align}
where $l_1 = \frac{U_2 -  U_1\beta}{\alpha(\alpha-\beta)}$ and $l_2 = \frac{- U_2 +  U_1\alpha}{\beta(\alpha-\beta)}$.

\vskip 0.3cm
There are two cases:
\begin{enumerate}
    \item $\sqrt{\Delta_{\mathbf U}}$ is not an integer and $C_{\mathbf U}$ is irreducible over $\Z$ or
    \item $\Delta_{\mathbf U}$ is a square of an integer and $C_{\mathbf U}$ is a product of two linear factors.
\end{enumerate}
These will yield two different results.

\subsection{$\Delta_{\mathbf U}$ is not a square of an integer}
For a prime number $p$ we denote by $\F_{p}$ the field of integers modulo $p$. If additionally $p$ is odd, then by $\Li{a}{p}$ we mean the Legendre symbol.

Let us consider a prime $p$ such that $\Li{\Delta_{\mathbf U}}{p}= -1$. Notice that the extension $\F_p(\alpha, \beta)$ contains both $\alpha$ and $\beta$. Therefore
\[
    \sqrt{\Delta_{\mathbf U}}^p \equiv
    \Delta_{\mathbf U}^{\frac{p-1}{2}} \cdot \sqrt{\Delta_{\mathbf U}} \equiv 
    \Li{\Delta_{\mathbf U}}{p} \cdot \sqrt{\Delta_{\mathbf U}} \equiv
    -\sqrt{\Delta_{\mathbf U}} \pmod p
\]
Thus, for any integers $A,B$ we have
\[
    \left(A+B\sqrt{\Delta_{\mathbf U}}\right)^p \equiv
    A^p + B^p\sqrt{\Delta_{\mathbf U}}^p \equiv
    A - B\sqrt{\Delta_{\mathbf U}}
    \pmod{p}
\]

Consequentially, the transformation $x \mapsto x^p$ permutes the roots modulo $p$, that is $\alpha^p\equiv\beta\pmod p$ and $\beta^p \equiv \alpha\pmod p$.

\begin{theorem}\label{TheoremNotSquare}
    The following conditions are equivalent:
    \begin{enumerate}
        \itemsep-0.2em
        \item the sequence ${\mathbf U}$ almost satisfies the Dold condition, 
        \item $l_1=l_2$, 
        \item the sequence $2r_2\rad(\Delta_\mathbf{U}){\mathbf U}$ satisfies the Dold condition. 
    \end{enumerate}
\end{theorem}

\begin{proof}

    $(1\Rightarrow 2)$

    If the sequence ${\mathbf U}$ is almost realizable, then $p|U_p-U_1$ for all but finitely many prime numbers $p$. We take a prime number $p$ such that $p|U_p-U_1$, $\left(\frac{\Delta}{p}\right)=-1$ and $p$ does not divide denominators of $l_1$ and $l_2$. There are infinitely many such primes. Notice that:
    \begin{equation}
      \begin{split}
            0 &\equiv
            U_p - U_1 \equiv
            (l_1\alpha^p + l_2\beta^p) - (l_1\alpha + l_2\beta) \\ 
            &\equiv 
            (l_1\beta + l_2\alpha) - (l_1\alpha + l_2\beta) \equiv
            (l_1-l_2)(\beta - \alpha) \pmod p
      \end{split}
      \end{equation}
    Therefore, $\alpha \equiv \beta \pmod p$ or $l_1 \equiv l_2 \pmod p$. This property holds for infinitely many prime numbers $p$. If $\alpha \equiv \beta \pmod p$ for infinitely many numbers $p$, then  $\alpha = \beta$, but then the quadratic equation has only one solution, hence $\Delta_{\mathbf U} = 0$, which contradicts the assumption that $\Delta_{\mathbf U}$ is not a square of an integer. Therefore, $l_1 \equiv l_2 \pmod p$ must be true for infinitely many prime numbers $p$, which implies that $l_1=l_2$.

    \vskip 0.5cm
    \noindent $(2\Rightarrow3)$

    The coefficients $l_1 = l_2$ may not be integers. To be able to analyze them modulo prime $p$ we must first make sure that $p$ does not divide their denominators. Fortunately, one can notice that $2r_2l_1$ is an integer because
    \begin{align}
        2r_2l_1 = r_2(l_1+l_2) = r_2U_0 = U_2 - r_1U_1 \in \Z ,
    \end{align}
    where the second equality comes from the exact formula and the third one from the recurrence relation.

    We will show that the sequence $\rad(\Delta_\mathbf{U})\mathbf{V}$ satisfies the Dold condition where ${V}_n := \alpha^n+\beta^n$. This will imply that $2r_2\rad(\Delta_\mathbf{U})\mathbf{U}$ also satisfies this condition as this sequence is an integer multiple of $\rad(\Delta_\mathbf{U})\mathbf{V}$. 

    To show that $\rad(\Delta_\mathbf{U})\mathbf{V}$ satisfies the condition we will prove that for every prime $p$ we have
    \begin{align}\label{the_not_sq_e1}
        \begin{split}
        \rad(\Delta_\mathbf{U})\left(\alpha^{sp^k}+\beta^{sp^k}\right) \equiv 
        \rad(\Delta_\mathbf{U})\left(\alpha^{sp^{k-1}}+\beta^{sp^k}\right) 
        \pmod {p^k} 
        \\ \text{for all $k, s\in \N_+$}
        \end{split}
    \end{align}

For any prime $p$ the expression $\Delta_\mathbf{U}^\frac{p-1}{2} \pmod p$ (and consequently $\Li{\Delta_\mathbf{U}}{p}$) can only take on one of three values. 
\begin{enumerate}
    \item $\Delta_\mathbf{U}^\frac{p-1}{2} \equiv -1 \pmod p$ and $\Delta_\mathbf{U}$ is not a quadratic residue. We already considered this case and saw that
    $\alpha^p\equiv\beta\pmod p$ and $\beta^p \equiv \alpha\pmod p$. By Lemma we have \ref{lem_allInOne}  $\alpha^{p^k}\equiv\beta^{p^{k-1}}\pmod {p^k}$ and $\beta^{p^k}\equiv\alpha^{p^{k-1}}\pmod {p^k}$. Taking the last two equivalences to the power of $s$ and adding them together we get (\ref{the_not_sq_e1}).

    \item $\Delta_\mathbf{U}^\frac{p-1}{2} \equiv 1 \pmod p$ and $\Delta_\mathbf{U}$ is a quadratic residue. This case follows similarly. We have
     \begin{align}
         \left(A+B\sqrt{\Delta_{\mathbf U}}\right)^p \equiv
         A^p + B^p\cdot \Delta_{\mathbf U}^{\frac{p-1}{2}} \cdot \sqrt{\Delta_{\mathbf U}} \equiv
         A + B\sqrt{\Delta_{\mathbf U}}
         \pmod{p}.
    \end{align}
    Therefore $\alpha^p\equiv\alpha\pmod p$ and $\beta^p \equiv \beta\pmod p$.  By Lemma \ref{lem_allInOne}  $\alpha^{p^k}\equiv\alpha^{p^{k-1}}\pmod {p^k}$ and $\beta^{p^k}\equiv\beta^{p^{k-1}}\pmod {p^k}$. Again, taking both to the power of $s$ and adding them we get $(\ref{the_not_sq_e1})$.

    \item $\Delta_\mathbf{U}^\frac{p-1}{2} \equiv 0 \pmod p$ and $p|\Delta_\mathbf{U}$. We have
     \begin{align}
         \left(A+B\sqrt{\Delta_{\mathbf U}}\right)^p \equiv
         A^p + B^p\sqrt{\Delta_{\mathbf U}}^p \equiv
         A
         \pmod{p}.
    \end{align}
    Hence there must exist $A\in\Z$ such that $\alpha^p\equiv\beta^p\equiv A\pmod p$. Again by Lemma \ref{lem_allInOne} we have $\alpha^{p^k}\equiv\beta^{p^k}\equiv A^{p^{k-1}}\pmod {p^k}$. We can see that
    \begin{align}
        \alpha^{p^ks} \equiv  
        \left(   \alpha^{p^k} \right)^s  \equiv 
         \left( A^{p^{k-1}}\right)^s  \equiv
         \left( A^{p^{k-2}}\right)^s  \equiv
         \\\equiv
         \left(\alpha^{p^{k-1}}\right)^s  \equiv
          \alpha^{p^{k-1}s} \pmod {p^{k-1}}
    \end{align}
    for any $s\in\N_+$. The second equivalence can be proven by applying Lemma \ref{lem_allInOne} to $A^p\equiv A\pmod p$ (which is true by the Fermat's little theorem). The similar argument works to show that  $\beta^{p^k} \equiv   \beta^{p^{k-1}s} \pmod {p^{k-1}}$. Adding these results and multiplying them by $p$ we get  $ p\left(\alpha^{sp^k}+\beta^{sp^k}\right) \equiv  p\left(\alpha^{sp^{k-1}}+\beta^{sp^k}\right) \pmod {p^k}$. Because in this case $p|\rad(\Delta_\mathbf{U})$ this implies (\ref{the_not_sq_e1}). 

\end{enumerate}




    

    \noindent $(3\Rightarrow1)$ Trivial.
\end{proof}

\begin{example}
    The sequence $\mathbf{U}$ defined by recursion as $U_{n+2} = 12U_{n+1}+3U_n$ for $n\in \N_+$ and $U_1 = 2$, $U_2 = 25$ has the exact formula 
    \begin{align}
        U_n = \frac{\left(6+\sqrt{39}\right)^n+\left(6-\sqrt{39}\right)^n}{6}
    \end{align}
    By the theorem above $\mathbf{U}$ almost satisfies the Dold condition and $\fail(\mathbf{U})$ divides $2r_2\rad(\Delta_\mathbf{U}) = 2\cdot 3\cdot 78 = 2^2\cdot 3^2\cdot13$. One can verify that $U_3 = 306$. We now notice that $2\not|\ U_2-U_1 = 25-2 = 23$ and $3\not|\ U_3-U_1 = 306-2 = 304$ which implies that $2\cdot 3=6\ |\ \fail(\mathbf{U})$. One can also calculate that $U_{13} \equiv 2 \pmod 13$ and that $13|U_{13}-U_1$. We cannot therefore verify using this simple argument whether $13|\fail(\mathbf{U})$.   
\end{example}

\subsection{$\Delta_{\mathbf U}$ is a square of an integer}

In this case the roots $\alpha$ and $\beta$ are integers. 

\begin{theorem} \label{theor_rad}
    The sequence $r_2\rad(\Delta_{\mathbf U})\mathbf U$ satisfies the Dold condition
\end{theorem}
\begin{proof}
    By Lemma \ref{lem_betterDef} it is enough to show that $r_2\rad(\Delta_{\mathbf U}) U_{sp^k} \equiv r_2\rad(\Delta_{\mathbf U}) U_{sp^{k-1}} \pmod{p^k}$ for every prime $p$, and $s,k\in\N_+$.
    
    If $\Delta_{\mathbf U}$ is not divisible by $p$, then the denominator of $r_2l_1 = \alpha\beta\cdot \frac{U_2 - U_1\beta}{\alpha(\alpha-\beta)} = \frac{\beta(U_2 - U_1\beta)}{(\alpha-\beta)}$ is not divisible by $p$. The same happens in the case of $r_2l_2$. We can write the following
    \begin{align}
        & r_2U_{sp^k} - r_2U_{sp^{k-1}} \\&\equiv 
        ((r_2l_1)\alpha^{sp^k}+(r_2l_2)\beta^{sp^k}) -  ((r_2l_1)\alpha^{sp^{k-1}}+(r_2l_2)\beta^{sp^{k-1}})  \\&\equiv 
        (r_2l_1)(\alpha^{sp^k} - \alpha^{sp^{k-1}}) +(r_2l_2)(\beta^{sp^k} -\beta^{sp^{k-1}}) \\&\equiv 
        (r_2l_1)\cdot 0 + (r_2l_2)\cdot 0  \\&\equiv 
        0
        \pmod {p^k}
    \end{align}

    The penultimate equivalence is true because  $A^{p^k}\equiv A^{p^{k-1}} \pmod {p^k}$ for every integer $A$ and prime $p$. This shows that $r_2U_{sp^k} - r_2U_{sp^{k-1}} \equiv 0 \pmod{p^k}$, which implies that $r_2\rad(\Delta_{\mathbf U}) U_{sp^k} \equiv r_2\rad(\Delta_{\mathbf U}) U_{sp^{k-1}} \pmod{p^k}$ for every $s,k\in \N_+$ and prime $p\not|\ \Delta_\mathbf{U}$.

    The problem arises when $p$ is a divisor of $\Delta_{\mathbf U}$. Let $n:=\nu_p(\sqrt{\Delta_{\mathbf U}}) = \nu_p(\alpha-\beta) > 0$ be the $p$-adic valuation of $\sqrt{\Delta_{\mathbf U}}$. The denominators of $r_2p^nl_1$ and $r_2p^nl_2$ are therefore not divisible by $p$. For simplicity we will write $t_1 = r_2p^nl_1$ and $t_2 = r_2p^nl_2$. 
    
    We can see that $\alpha - \beta = \sqrt{\Delta_{\mathbf U}} \equiv 0 \pmod{p^n}$, so $\alpha \equiv \beta \pmod{p^n}$. Using Lemma \ref{lem_allInOne} we conclude that
    \begin{align} \label{nn1}
        \alpha^{p^{k}} \equiv \beta^{p^k} \pmod{p^{k+n}}.
    \end{align}
    Also by Lemma we have \ref{lem_allInOne} $\alpha^{p^{k-1}} \equiv \beta^{p^{k-1}} \pmod{p^{k+n-1}}$. Therefore
    \begin{align}\label{nn2}
        \alpha^{p^{k-1}} \equiv \beta^{p^{k-1}} +zp^{n+k-1} \pmod{p^{k+n}}
        \quad\quad\text{for some $z\in\Z$}.
    \end{align}
     
    Consider the following difference.
    \begin{equation}
        \begin{split}
        &r_2p^nU_{sp^k} - r_2p^nU_{sp^{k-1}}  
        \\ &\equiv (t_1\alpha^{sp^k}+t_2\beta^{sp^k}) -  (t_1\alpha^{sp^{k-1}}+t_2\beta^{sp^{k-1}}) 
        \\ &\equiv
        t_1(\alpha^{sp^k} - \alpha^{sp^{k-1}}) +t_2(\beta^{sp^k} -\beta^{sp^{k-1}})
        \\ &\!\!\!\!\overset{(\ref{nn1}),(\ref{nn2})}{\equiv}
        t_1(\beta^{sp^k} - (\beta^{sp^{k-1}}+p^{n+k-1}z)) +t_2(\beta^{sp^k} -\beta^{sp^{k-1}}) 
        \\ &\equiv
        (t_1+t_2)(\beta^{sp^k} - \beta^{sp^{k-1}}) -t_1(p^{n+k-1}z)
        \\ &\equiv
        p^n(r_2U_0)(\beta^{sp^k} - \beta^{sp^{k-1}}) -t_1(p^{n+k-1}z)
        \\ &\equiv
        -t_1(p^{n+k-1}z)  \pmod{p^{k+n}}
    \end{split}
    \end{equation}

    The last equivalence is true because by Lemma \ref{lem_allInOne} we have $\beta^{p^{k}} \equiv \beta^{p^{k-1}} \pmod{p^k}$ and  $r_2U_0$ is an integer, so $p^{k+n}\mid p^n(r_2U_0)(\beta^{p^{k}} -\beta^{p^{k-1}})$.

    The result of above calculations is
    $$r_2p^nU_{sp^k} - r_2p^nU_{sp^{k-1}} \equiv -t_1(p^{n+k-1}z)  \pmod{p^{k+n}}.$$
    First dividing by $p^n$ and then multiplying by $p$ yields
    \begin{align}
        r_2U_{sp^k} - r_2U_{sp^{k-1}} \equiv -(t_1)(p^{k-1}z) \pmod{p^{k}}, \\
         pr_2U_{sp^k} - pr_2U_{sp^{k-1}} \equiv -(t_1)(p^{k}z) \equiv 0 \pmod{p^{k}}.
    \end{align}

    Hence, for each  prime $p$ the equivalence $r_2\rad(\Delta_{\mathbf U})U_{sp^k} \equiv r_2\rad(\Delta_{\mathbf U})U_{sp^{k-1}}\pmod{p^k}$ is true. This means that $r_2\rad(\Delta_{\mathbf U}){\mathbf U}$ satisfies the Dold condition.
\end{proof}

The above theorem proves two interesting facts. The first one is that the sequence ${\mathbf U}$ is always almost realizable no matter the choice of $l_1$ and $l_2$ (equivalently: no matter the choice of $U_1$ and $U_2$) which is not true in the former case. The second fact is that it gives a bound for the repairing factor of ${\mathbf U}$ and in fact this is the best possible bound in terms of $\Delta_{{\mathbf U}}$. More formally:

\begin{theorem}
    For every $\delta\in\N_+$ there exists a sequence ${\mathbf U}$ given by linear recursion of order 2 with $\sqrt{\Delta_{\mathbf U}} = \delta$ such that  $\rad(\Delta_{\mathbf U})| \fail({\mathbf U})$.
\end{theorem}
\begin{proof}
    The sequence satisfying the statement of the theorem is
    \begin{align}
        U_n = (\delta+2)U_{n-1} - (\delta+1)U_{n-2} \quad\text{for $n\geqslant 2$ \ and\ \ $U_0=1$, $U_1=\delta$} \label{dd}
    \end{align}
    Its characteristic polynomial and closed form are as follows:
    \begin{align}
        C_U &= (x-1)(x-1-\delta)\\
        U_n &= \frac{1}{\delta} + \frac{\delta-1}{\delta}(\delta+1)^n
    \end{align}

    Notice that $\sqrt{\Delta_{\mathbf U}} = (\delta+1) - (1) = \delta$. 
    
    We will show that $p\nmid U_p-U_1$ for every $p|\delta$. Since the condition $p|U_p-U_1$ is necessary for ${\mathbf U}$ to be realizable we will need to repair it by at least a factor of $p$ so $p|\fail({\mathbf U})$. This is enough to conclude that $\rad(\Delta_{\mathbf U}) = \rad(\delta)|\fail({\mathbf U})$.
    
    Let $p|\delta$.
    \begin{align}
        U_p - U_1 &=  
        \left( \frac{1}{\delta} + \frac{\delta-1}{\delta}(\delta+1)^p \right) - \left( \frac{1}{\delta} + \frac{\delta-1}{\delta}(\delta+1) \right) \\ 
        &=  \frac{\delta-1}{\delta} \left( (\delta+1)^p - (\delta+1) \right) \\
        &=  \frac{\delta-1}{\delta} \left( \delta^p+1 + p\delta r - (\delta+1) \right) \\
        &=  (\delta-1) \left( \delta^{p-1} + p r - 1 \right)
    \end{align}
    for some $r\in \Z$. Because $p|\delta$, we have
    \begin{align}
        U_p - U_1 \equiv 
         (\delta-1) \left( \delta^{p-1} + p r - 1 \right) 
         \equiv (-1) (-1) \equiv 1 \pmod{p}.
    \end{align}
    Clearly $p\nmid U_p - U_1$ which concludes the proof.

\end{proof}

\subsection{The Dold condition for the sequence  $(U_{n^{2t}})_{n\in\N_+}$}

While not every sequence ${\mathbf U}$ defined by linear recursion of order 2 satisfies the Dold condition, it might be surprising that when we take a subsequence consisting of only indices that are square numbers, then the resulting sequence $(U_{n^2})_{n\in\N_+} := (U_1,  U_4,  U_9, \ldots)$ is always almost realizable no matter the reducibility of the characteristic polynomial.
\begin{theorem}\label{theor_n2}
    The sequence $(U_{n^2})_{n\in\N_+}$ almost satisfies the Dold condition
\end{theorem}

\begin{proof}
It will be enough to show that 
\begin{align} \label{theor_n2_e1}
\raddel\left(\alpha^{{(sp^k)}^2} - \alpha^{{(sp^{k-1})}^2}\right)  &\equiv 0 \pmod{p^k} \\
\raddel\left(\beta^{{(sp^k)}^2} - \beta^{{(sp^{k-1})}^2}\right)  &\equiv 0 \pmod{p^k} \label{theor_n2_e2}
\end{align}
for all primes $p$ and $s, k \in \N_+$ because then 
\begin{align}
    & N\raddel U_{{(sp^k)}^2} - N\raddel{U}_{{(sp^{k-1})}^2} \\ &\equiv
    \raddel\left((Nl_1)\alpha^{{(sp^k)}^2} + (Nl_2)\beta^{{(sp^k)}^2}\right)
    - \raddel\left((Nl_1)\alpha^{{(sp^{k-1})}^2} + (Nl_2)\beta^{{(sp^{k-1})}^2}\right) \\ &\equiv
    (Nl_1)\raddel\left(\alpha^{{(sp^k)}^2} - \alpha^{{(sp^{k-1})}^2}\right)
    + (Nl_1)\raddel\left(\beta^{{(sp^k)}^2} - \beta^{{(sp^{k-1})}^2}\right)
    \\ &\overset{(\ref{theor_n2_e1}), (\ref{theor_n2_e2})}{\equiv} (Nl_1) \cdot 0 + (Nl_2) \cdot 0  
    \\&\equiv 0 \pmod{p^k}
\end{align}
where $N$ is such an integer that $Nl_1$ and $Nl_2$ are both algebraic integers. This shows that the sequence $(N\raddel {U}_{n^2})_{n\in\N_+}$ satisfies the Dold condition. One can see that $N$ can be set to $r_2\Delta_\mathbf{U}$. This gives the bound $\fail(\mathbf{U}) | r_2\Delta_\mathbf{U}\rad(\Delta_\mathbf{U})$.

The only thing left is to show (\ref{theor_n2_e1}). The congruence (\ref{theor_n2_e2}) will follow analogously.

Fix a prime $p$. While proving Theorem \ref{TheoremNotSquare} we saw that if $\Delta_\mathbf{U}^\frac{p-1}{2} \equiv -1 \pmod p$ then $\alpha^{p^2} \equiv \beta^p \equiv \alpha \pmod p$. If  $\Delta_\mathbf{U}^\frac{p-1}{2} \equiv 1 \pmod p$ then we get the same result $\alpha^{p^2} \equiv \alpha^p \equiv \alpha \pmod p$. In these two cases using again Lemma \ref{lem_allInOne}
 we can see that $\alpha^{p^{2+t}} \equiv \alpha^{p^t} \pmod{p^{t+1}}$ for any $t \in \N$. Now rising both sides to the power of $s^2$ and substituting $t = 2k-2$ (notice $k\geqslant 1$) we get $\alpha^{{(p^{k}s)}^2} \equiv \alpha^{{(sp^{k-1})}^2} \pmod{p^{2k-1}}$. This finally implies (\ref{theor_n2_e1}) as $2k-1 \geqslant k$.
 
 The last case is when $\Delta_\mathbf{U}^\frac{p-1}{2} \equiv 0 \pmod p$. Notice that we only need to prove (\ref{theor_n2_e1}) for $k\geqslant 2$ because as $p|\Delta_\mathbf{U}$ the case for $k=1$ is trivial.  We saw in the proof of Theorem \ref{TheoremNotSquare} that in this case $\alpha^p \equiv A \pmod p$ for some $A \in \Z$. By Lemma \ref{lem_allInOne} we have
 \begin{align}\label{theor_n2_e3}
     \alpha^{p^2} \equiv A^p \pmod {p^2}.
 \end{align}
 Now using Fermat's little theorem and Lemma \ref{lem_allInOne} on steps marked with ($*$) we see that
 \begin{align}
     \alpha^{p^4} \equiv
     \left( \alpha^{p^2}\right)^{p^2} \overset{(\ref{theor_n2_e3})}{\equiv} 
     \left( A^p\right)^{p^2}
     \equiv A^{p^3}
    \overset{(*)}{\equiv} A^{p^2}
    \overset{(*)}{\equiv}A^p
       \overset{(\ref{theor_n2_e3})}{\equiv} \alpha^{p^2}
       \pmod{p^2}.
 \end{align}
 With that it is easy to prove the congruence (\ref{theor_n2_e1}) for $k\geqslant 2$ in a similar way to the previous two cases.

\end{proof}

This shows that the sequence $({U}_{n^{t}})_{n\in\N_+}$ almost  satisfies the Dold condition for $t=2$. But what about the other powers?

One can see that if a sequence $\mathbf{V}$ satisfies the Dold condition then does so the sequence $({V}_{n^t})_{n\in\N_+}$ for any $t\in\N_+$ because
\begin{align}
    {V}_{(p^ks)^t} \equiv 
    {V}_{p^{kt}s^t} \equiv
    {V}_{p^{kt-1}s^t} \equiv
    \cdots \equiv
    {V}_{p^{kt-t}s^t} \equiv
    {V}_{(p^{k-1}s)^t}
    \pmod {p^{kt-t+1}}.
\end{align}

As $kt-t+1 \geqslant k$ we have ${V}_{(p^ks)^t}\equiv {V}_{(p^{k-1}s)^t} \pmod{p^k}$ for any prime $p$ and $k,s, t \in \N_+$ which proves that $({V}_{n^t})_{n\in\N_+}$ satisfies the Dold condition. 

The same argument works to show that if a sequence $\mathbf{V}$ \textit{almost}
satisfies the Dold condition then so does the sequence $({V}_{n^t})_{n\in\N_+}$ for any $t\in \N_+$.

\begin{corollary}\label{theor_n2_final}
    Let $\mathbf{U}$ be any recurrent sequence of order 2. Then the sequence $(U_{n^{t}})_{n\in\N_+}$ almost satisfies the Dold condition for all even $t\in\N_+$.
\end{corollary}

\section{The case of arbitrary order}

The proofs in the case of order 2 were straightforward. This case had very limited number of unknown variables, only 2. They were $\alpha$ and $\beta$ or equivalently  $r_1,$ and  $r_2$. We exploited this fact thoroughly. The general case is a bit more tricky to show. We will take a step by step approach beginning with the simplest but most restricted case. 

\subsection{Sequences that are \nice}

We begin with narrowing our focus to certain class of sequences that behave nicely.
 
\begin{defin}\label{def_conv}
    We will call a sequence ${\mathbf U}$ \nice if
    
    \begin{enumerate}
    \item[(c1)] ${\mathbf U}$ is defined by linear recursion as in (\ref{def_Un}),
    \item[(c2)] its characteristic polynomial $C_{\mathbf U}$ is irreducible over $\F_p$ for infinitely many primes $p$.
\end{enumerate}

\end{defin}



\begin{lemma}  \label{theor_Un}
    If a sequence ${\mathbf U}$ satisfies the Dold condition and is \nice\!\!, then $l_0=l_1=\cdots=l_{d-1}$.
\end{lemma}

\begin{proof}

    There are infinitely many prime numbers that are not ramified in $\OO_K$ and for which (c2) occurs. Let  $\pp$  be a prime ideal lying over one of them.
    Because ${\mathbf U}$ satisfies the Dold condition, we have $ U_{p^ks} \equiv  U_{p^{k-1}s} \pmod \pp$ for every $k,s\in\N_+$. By chaining these equivalences we can get
    \begin{align}\label{theor_Un_e1}
         U_{s} \equiv  U_{p^{k}s} \pmod \pp \quad\text{for every $k,s\in\N_+$}.
    \end{align}
    By \cite{frob} (Lemma 5.5.8) because $C_\mathbf{U}$ is irreducible over $\F_p$, the automorphism $x \mapsto x^p$ permutes the roots $\alphas$ in a cyclic way. Therefore 
    \begin{align}\label{theor_Un_e1.5}
        \sum_{i=0}^{d-1}\alpha_j^{p^i} \equiv \sum_{i=0}^{d-1}\alpha_i \pmod \pp 
        \quad\quad\text{for each $0\leqslant j<d$}.
    \end{align}
    We have
    \begin{equation}
      \begin{split}
        d\ U_s \overset{(\ref{theor_Un_e1})}{\equiv}
        \sum_{i=0}^{d-1}U_{p^is} \equiv
        \sum_{i=0}^{d-1}  \sum_{j=0}^{d-1} l_j\alpha_{j}^{p^is} \equiv
        \sum_{j=0}^{d-1} l_j\sum_{i=0}^{d-1}  \alpha_{j}^{p^is} 
        \\ \overset{(\ref{theor_Un_e1.5})}{\equiv}
        \sum_{j=0}^{d-1} l_j\sum_{i=0}^{d-1}  \alpha_{i}^s \equiv 
        \sum_{j=0}^{d-1} l_j \cdot \sum_{i=0}^{d-1}  \alpha_{i}^s \equiv
        U_0\sum_{i=0}^{d-1}  \alpha_{j}^s \pmod \pp.
    \end{split}  
    \end{equation}

    The congruence above is true for infinitely many primes $\pp$, hence $d U_s= U_0\sum_{i=0}^{d-1} \alpha_{i}^s$. Since the choice of $s\in\N_+$ was arbitrary we see that the sequence ${\mathbf U}$ has the same values at any term as the sequence ${\mathbf V}$ defined as $V_n := \frac{ U_0}{d}\sum_{i=0}^{d-1} \alpha_{i}^n$. We can also extend the sequences to be defined at 0 and see that ${V}_0 = \frac{ U_0}{d}d = U_0$

    Let us construct a $d\times d$ matrix as follows:
    \begin{align}
        M := \begin{bmatrix}
            \alpha_0^{0} & \alpha_1^{0} & \cdots & \alpha_{d-1}^{0}\\
            \alpha_0^{1} & \alpha_1^{1} & \cdots & \alpha_{d-1}^{1}\\
            \vdots & \vdots & \ddots & \vdots \\
            \alpha_0^{d-1} & \alpha_1^{d-1} & \cdots & \alpha_{d-1}^{d-1}
            \end{bmatrix}.
    \end{align}
    We notice that
    
    \begin{align}
        M \cdot 
        \begin{bmatrix}
            l_0\!-\!\frac{ U_0}{d} \\ 
            l_1\!-\!\frac{ U_0}{d}\\ 
            \vdots \\ 
            l_{d-1}\!-\!\frac{ U_0}{d}
        \end{bmatrix}
        =
        \begin{bmatrix}
             U_{0}- V_{0} \\ 
             U_{1}- V_{1}\\ 
            \vdots \\ 
             U_{d-1}- V_{d-1}
        \end{bmatrix}
        =
        \begin{bmatrix}
            0     \\ 
            0\\ 
            \vdots \\ 
            0
        \end{bmatrix}.
    \end{align}

    The matrix $M$ is a Vandermonde matrix, so its determinant is $\prod_{0\leqslant i<j<d}(\alpha_i-\alpha_j)$, which is nonzero. Then the matrix $M$ multiplied my any nonzero vector must be nonzero. Since the equation above is true, it is implied that the vector $[l_0\!-\!\frac{ U_0}{d},\ldots,l_{d-1}\!-\!\frac{ U_0}{d}]^T$ is zero, so $l_0=l_1=\cdots=l_{d-1}=\frac{ U_0}{d}$.

\end{proof}

The inverse to the statement of the above theorem is also true. If ${U}_n = l\sum_{i=0}^{d-1}\alpha^n$ for some constant $l$ and $\mathbf{U}$ is \nice then it almost satisfies the Dold condition. To show that we will need the following lemma. 

\begin{lemma}
    \label{lem_lte}
    Let $\beta_1, \cdots, \beta_n$ be all the roots of a polynomial with integer coefficients. Then 
    $$
        \sum_{i=1}^n \beta_i^{p^{k}} \equiv \sum_{i=1}^n \beta_i^{p^{k-1}}
        \pmod {\pp^{ek}} 
    $$
    for any prime $p$ and $k\in \N$.
   
\end{lemma}
\begin{proof}
    A partition of $n$ is a (unordered) multiset of positive integers $I = \{i_1, \ldots, i_m\}$ that sum up to $n$.
    A composition of $n$ is a (ordered) sequence of positive  integers $J = (i_1, \ldots, i_m)$ that sum up to $n$. When we look at a composition $J$ of $n$ and discard the order of the elements we get some partition $I$. We will denote this relation with $J\sim I$. Both partitions and compositions play a crucial role in the following formula.

    Let us define $S$ is the set of all partitions of $p$ into $n$ elements and $S^-$ is $S$ without the trivial partition $p+0+\cdots+0$.
    Consider a prime power of a sum:
    \begin{align}\label{theor_real_easy_sum1}
        &(\beta_1 + \beta_2 +\cdots+ \beta_n)^p 
        \\&\equiv 
        \sum_{i_1+\ldots+i_n=p} \binom{p}{i_1, \cdots, i_n} \left(\beta_1^{i_1} \cdots \beta_n^{i_n}\right) \label{theor_real_easy_sum1.3}\\&\equiv
         \sum_{I=\{i_1,\ldots,i_n\}\in S} \binom{p}{i_1, \cdots, i_n} \sum_{\substack{J=(j_1,\ldots, j_n) \\ J \sim I}}\left(\beta_1^{i_1} \cdots \beta_n^{i_n}\right) 
         \label{theor_real_easy_sum1.6}\\&\equiv
        \sum_{i=1}^n \beta_i^{p} + \sum_{I=\{i_1,\ldots,i_n\}\in S^-} \binom{p}{i_1, \cdots, i_n} 
          \sum_{\substack{J=(j_1,\ldots, j_n) \\ J \sim I}}\left(\beta_1^{j_1} \cdots \beta_n^{j_n}\right).\label{theor_real_easy_sum2}
    \end{align}
    The sum in line (\ref{theor_real_easy_sum1.3}) sums over all compositions of $p$, the outer sum in line (\ref{theor_real_easy_sum1.6}) over all partitions and the inner sum over all compositions with the given partition. The sums in line (\ref{theor_real_easy_sum2}) work analogusly. 
    Notice that the value of the innermost sum is an integer as it is a symmetric polynomial of $\beta_1, \ldots, \beta_n$. The value of every term of the outer sum is an integer divisible by $p$. This is guaranteed by the divisibility of the multinomial symbol $\binom{p}{i_1, \cdots, i_n}$ by $p$ when $i_1,\ldots ,i_n$ are less than $p$.
    
    We will prove the lemma by induction. The base case for $k=1$ goes as follows. As the sum $\sum_{i=1}^n \beta_i$ is an integer we can use the Fermat's Little Theorem.
    \begin{align}\label{theor_real_easy_e1}
        \sum_{i=1}^n \beta_i &\equiv \left(\sum_{i=1}^n \beta_i\right)^p\pmod {\pp^e}
    \end{align}
    On the other hand,  
    \begin{equation}
        \begin{split}\label{theor_real_easy_e2}
        &\left(\sum_{i=1}^n \beta_i\right)^p 
        \equiv 
        \sum_{i=1}^n \beta_i^p + \sum_{I=\{i_1,\ldots, i_n\}\in S^-} \binom{p}{i_1, \ldots, i_n} \sum_{\substack{J=(j_1,\ldots, j_n) \\ J \sim I}}(\beta_1^{i_1}\cdots\beta_n^{i_n}) 
        \\&\equiv 
        \sum_{i=1}^n \beta_i^p  + 0
        \pmod {\pp^e} . 
        \end{split}
    \end{equation}

    Connecting congruences (\ref{theor_real_easy_e1})  and (\ref{theor_real_easy_e2}) one can see that $\sum_{i=1}^n \beta_i \equiv  \sum_{i=1}^n \beta_i^p  \pmod {\pp^e}$ which ends the base case. For the step case we assume that the Theorem holds for any suitable choice of $\beta_1,\ldots,\beta_n$ and some $k$ and we will prove it for $k+1$.

    Firstly by the Lemma \ref{lem_allInOne} we have
    
    \begin{align} \label{theor_real_easy_e5}
        \sum_{i=1}^n \beta_i^{p^{k}} \equiv \sum_{i=1}^n \beta_i^{p^{k-1}} \pmod{\pp^{ek}} 
        \quad\Rightarrow\quad
        \left(\sum_{i=1}^n \beta_i^{p^{k}}\right)^p \equiv \left(\sum_{i=1}^n \beta_i^{p^{k-1}}\right)^p \pmod{\pp^{e(k+1)}}.
    \end{align}

    Now we apply the congruences (\ref{theor_real_easy_sum1})- (\ref{theor_real_easy_sum2}) to the expression $\left(\sum_{i=1}^n \beta_i^{p^{ek}}\right)^p$.

    \begin{align}
        \left(\sum_{i=1}^n \beta_i^{p^{k}}\right)^p 
        -\sum_{i=1}^n \beta_i^{p^{k+1}}  &\equiv \\
          \sum_{I = \{i_1,\ldots,i_n\}\in S^-} \binom{p}{i_1, \cdots, i_n} 
         \sum_{\substack{J=(j_1,\ldots, j_n) \\ J \sim I}}\left(\beta_1^{p^k}\right)^{i_1} \cdots \left(\beta_n^{p^k}\right)^{i_n} &\equiv\\
           \sum_{I = \{i_1,\ldots,i_n\}\in S^-} \binom{p}{i_1, \cdots, i_n} 
         \sum_{\substack{J=(j_1,\ldots, j_n) \\ J \sim I}}\left(\beta_1^{i_1} \cdots \beta_n^{i_n}\right)^{p^k} \label{theor_real_easy_e4}&\equiv\\
          \sum_{I = \{i_1,\ldots,i_n\}\in S^-} \binom{p}{i_1, \cdots, i_n} 
         \sum_{\substack{J=(j_1,\ldots, j_n) \\ J \sim I}}\left(\beta_1^{i_1} \cdots \beta_n^{i_n}\right)^{p^{k-1}}&\equiv\\
          \left(\sum_{i=1}^n \beta_i^{p^{k-1}}\right)^p 
        -\sum_{i=1}^n \beta_i^{p^{k}} 
        \pmod{\pp^{e(k+1)}}
    \end{align}

    The terms in the most inner sum of expression (\ref{theor_real_easy_e4}) are all the roots of some polynomial $W$. The coefficients of $W$ are symmetric polynomials in $\beta_i$, therefore by the fundamental theorem of symmetric polynomials (\cite{mcdonald} pages 19-21) they can be represented in terms of elementary symmetric polynomials $e_i(\beta_1, \ldots, \beta_n)$ which are integers by the definition of the numbers $\beta_i$, $i\in\{1,\ldots ,n\}$. We can therefore apply the inductive assumption. The assumption works under modulo $\pp^{ek} $  but because the multinomial symbols are divisible by $p$ we get the congruence modulo $\pp^{e(k+1)}$.
    
    Looking at implication (\ref{theor_real_easy_e5}) we can conclude that \linebreak $\sum_{i=1}^n \beta_i^{p^{k+1}} \equiv \sum_{i=1}^n \beta_i^{p^{k}} \pmod{\pp^{e(k+1)}}$.

\end{proof}
\begin{corollary}\label{c_sum}
    Let $U_n = l\sum_{i=0}^{d-1}\alpha_i^n$ be a sequence with irreducible characteristic polynomial. Then $\mathbf{U}$ almost satisfies the Dold conditoin.
\end{corollary}
\begin{proof}
    Let $N\in\Z$ be such a number that $Nl\in\Z$. With the above Lemma it is easy to see that $\mathbf{U}$ almost satisfies the Dold condition and that $\fail(\mathbf{U})|N$ because

\begin{align}
    N{U}_{sp^k} \equiv 
    (Nl)\sum_{i=0}^{d-1}\alpha_i^{sp^k}  \overset{\text{Lemma \ref{lem_lte}}}{\equiv}
    (Nl)\sum_{i=0}^{d-1}\alpha_i^{sp^{k-1}}  \equiv
    N{U}_{sp^{k-1}}  \pmod { \pp^{ek}}
\end{align}

for every $s,k\in \N_+$ and prime $p$ which implies that $N\mathbf{U}$ satisfies the Dold condition.  

\end{proof}

Collecting the results of this section, we get the following.

\begin{theorem}\label{the_nice}
    If a sequence $\mathbf{U}$ is \nice then it almost satisfies the Dold condition if and only if $l_0 = l_1 = \cdots = l_{d-1}$.
\end{theorem}

\subsection{The case of irreducible characteristic polynomial}

 In this case we assume that the characteristic polynomial $C_\mathbf{U}$ is irreducible. This is weaker assumption than $\textbf{U}$ to be \nice as there are polynomials that are not reducible over $\Z$ but reduce over any prime $p$. One of such polynomials is
 \begin{align}
     x^4-10x+1 = (x-\sqrt{2}-\sqrt{3})(x-\sqrt{2}+\sqrt{3})(x+\sqrt{2}-\sqrt{3})(x+\sqrt{2}+\sqrt{3}).
 \end{align}

\begin{theorem}\label{the_part}
    If a the characteristic polynomial $C_\mathbf{U}$ is irreducible then $\mathbf{U}$ almost satisfies the Dold condition if and only if $l_0 = l_1 = \cdots = l_{d-1}$.
\end{theorem}
\begin{proof}[The proof of "$\Leftarrow$"]
    The corollary \ref{c_sum} shows the implication "$\Leftarrow$". 
\end{proof}

The implication "$\Rightarrow$" is a bit harder to prove and will require tools that we will develop in the next case. We will complete this proof then.

If we, for a moment, believe in the above result, then we can prove the following strong bound on the $\fail$ factor.

\begin{lemma}\label{lem_gcd}
    If $C_\mathbf{U}$ is irreducible then $\fail(\mathbf{U})|\gcd(r_1, 2r_2, \cdots, dr_d)$.
\end{lemma}
\begin{proof}

The sequence $\mathbf{U}$ has the form ${U}_n = l\sum_{i=0}^n\alpha_i^{d-1}$ where $l = \frac{{U}_1}{\sum_{i=0}^{d-1}\alpha_i} = \frac{{U}_1}{r_1}$. Hence, $l$ is rational. Let $l=\frac{r}{t}$ where $r$ and $t$ are coprime. Since $\mathbf{U}$ is sequence of integers, then $t$ must divide each term of $r\sum_{i=0}^{d-1}\alpha_i^n$ and because $r$ and $t$ are coprime, we have
\begin{align}
    t\ |\  \sum_{i=0}^{d-1}\alpha_i^n =: {V}_n \quad
    \text{}
\end{align}
 for every $n\in\N_+$. We already know that $\mathbf{V}$ satisfies the Dold condition, so clearly $\fail(\mathbf{U})|t$. Now, the only thing left is to prove the estimation on $t$.


Since $t|{U}_k$ for every $k$, then 
\begin{align}\label{lem_gcd_inf}
    t | \gcd\left({V}_1, {V}_2, \ldots \right)
\end{align}
The greatest common divisor of infinitely many terms is quite troublesome but fortunately we can reduce it to a one with $d$ terms. Using the recurrence formula ${V}_n = r_1{V}_{n-1}+ r_2{V}_{n-2} + \cdots + r_d {V}_{n-d} $ for $n\geqslant d$ we can see that $ {V}_{d+1}, {V}_{d+2}, \ldots$ are just linear combinations of ${V}_1,{V}_2, \ldots,{V}_d $ so their presence (\ref{lem_gcd_inf}) changes nothing. Hence,
\begin{align}\label{lem_gcd_fin}
    t | \gcd\left({V}_1, {V}_2, \ldots \right) = 
    \gcd\left({V}_1, {V}_2, \ldots, {V}_d \right)
\end{align}

We notice that $r_k$ is just the $k$-th elementary symmetric polynomial in $\alphas$ and the value of ${V}_k$ is the $k$-th power sum of $\alphas$. With this in mind we can apply the Newton identities:

\begin{align}
    kr_k = \sum_{i=1}^k(-1)^{i-1}r_{k-i}{V}_i 
    \quad\text{ for $1 \leqslant k \leqslant d$}
\end{align}
This implies that $kr_k$ can be represented as linear combination of ${V}_1 , \ldots {V}_k$ and the coefficient in this combination near ${V}_k$ is $(-1)^{k-1}r_0 = \pm 1$. For each $k$ we can therefore add suitable multiples of ${V}_i$ ($1\leqslant i<k$) to ${V}_k$ in the greatest common divisor without changing its value.

\begin{equation}
    \begin{split}\label{lem_gcd_e1}
    &\gcd\left({V}_1, {V}_2, \ldots,  {V}_{d-1}, {V}_d \right) \\
    =&\gcd\left({V}_1, {V}_2, \ldots,  {V}_{d-1}, dr_d \right) \\
    =&\gcd\left({V}_1, {V}_2, \ldots,  (d-1)r_{d-1}, dr_d \right) \\
    &\vdots \\
    =&\gcd\left(r_1,2r_2, \ldots,  (d-1)r_{d-1}, dr_d \right) 
\end{split}  
\end{equation}

To end the proof it is enough to connect all the results

\begin{align}
    \fail(\mathbf{U}) \ \  |\ \ 
    t \ \overset{(\ref{lem_gcd_fin})}{|} \ 
    \gcd\left({V}_1, {V}_2, \ldots, {V}_d \right) \overset{(\ref{lem_gcd_e1})}{=}
    \gcd\left(r_1,2r_2, \ldots, dr_d \right).
\end{align}
\end{proof}

\begin{example}
    The Sequence ${U}_{n} = 10{U}_{n-2} - {U}_{n-4}$  for $n>4$ has the already mentioned characteristic polynomial equal to $x^4-10x^2+1$. 
    Setting ${U}_1 = 0$, ${U}_2=5$, ${U}_3=0$ and ${U}_4=49$ yields the formula 
    $$
    {U}_n = 
        \frac{1}{4}\left(\sqrt{2}+\sqrt{3}\right)^n
        +\frac{1}{4}\left(\sqrt{2}-\sqrt{3}\right)^n
        +\frac{1}{4}\left(-\sqrt{2}+\sqrt{3}\right)^n
        +\frac{1}{4}\left(-\sqrt{2}-\sqrt{3}\right)^n
    $$
    The sequence $\mathbf{U}$ by Theorem \ref{the_part} almost satisfies the Dold condition and by the Lemma \ref{lem_gcd} its repairing factor must divide $\gcd(1\cdot0,2\cdot 10, 3\cdot0, 4\cdot1) = 4$. Because $2\not|{U}_2 - {U}_1 = 5-1$ we have $2\ |\ \fail(\mathbf{U})\ |\ 4$.

\end{example}

\subsection{Sequences with $\Delta_{\mathbf U}\not=0$}

Given a sequence $\mathbf{U}$, let us construct a graph $G_\mathbf{U}$ with elements form the set $\{0, \ldots, d-1\}$ as its vertices. Between two vertices $k$ and $l$ we put an edge with label $\pp$ for prime ideal $\pp$ if and only if 
\begin{align}
    \exists_{i\in\N}: \ \alpha_s^{p^i}  \equiv \alpha_t \pmod \pp.
\end{align} 
Edges are not directed and between two vertices there may exist infinitely many edges. 

With $s \sim_\pp t$ we will denote that there exist an edge between $s$ and $t$ with label $\pp$. The relation $\sim_\pp$ is therefore an equivalence relation for any $\pp$.

\begin{lemma}\label{lem_infinEdge}
    If a sequence ${\mathbf U}$  satisfies the Dold condition, $\Delta_\mathbf{U} \not=0 $ and there exist infinitely many edges between vertices $s$ and $t$ in the graph $G_\mathbf{U}$ then $l_s=l_t$. 
\end{lemma}
\begin{proof}
    Fix $s, t \in \{ 0, \ldots, d-1\}$. Let $P_0 := \{\pp: s\sim_\pp t, \pp \text{ not ramified} \}$ be infinite set of prime ideals of $\OO_K$. Each element of $P_0$ has associated the Frobenius automorphism $\Phi_\pp: \OO_K/\pp \to\OO_K/\pp$ with the property that $\Phi_\pp(x) \equiv x^p\pmod{\pp}$ for $x \in \OO_K$.
    
     $\Phi_\pp$ permutes the roots $\alphas$ in some way. There are finitely many such permutations and infinitely many primes, so there exist an infinite subset $P \subset P_0$ such that the Frobenius automorphism of each prime ideal of $P$ permutes the roots in the same way. Let $h:\{0,\ldots,d-1\}\to\{0,\ldots,d-1\}$ be this permutation.

    To conclude, the construction of $P$ ensures that:
    \begin{enumerate}
        \item[(P1)] $P$ is infinite,
        \item[(P2)] for every $\pp\in P$ we have $s\sim_\pp t$,
        \item[(P3)] $\alpha_i^p \equiv \alpha_{h(i)} \pmod \pp $ for all $\pp\in P$,\ $0\!\leqslant\! i\!<\! d$.
    \end{enumerate}

    Notice that by property $(P3)$ the relation $\sim_\pp$ is the same for every $\pp\in P$. 
    By $[i]$ we will denote the equivalence class of element $i$ with respect to to any relation $\sim_\pp$ for $\pp\in P$. The set $[i]$ can be also seen as a cycle of element $i$ in permutation $h$. By (P2) we know that $s$ and $t$ are in the same class i.e. $[s] = [t]$.

    Since $\textbf{U}$ satisfies the Dold condition, we have $U_{p^ks} \equiv U_{p^{k-1}s} \pmod \pp$ for every $k, s\in\N_+$. By chaining these equivalences we get 
    \begin{align}\label{lem_infinEdge_e2}
        U_{s} \equiv U_{p^{k}s} \pmod \pp .
    \end{align}
    Let $e:=\lcm\left(|[0]|, |[1]|, \ldots,  |[d-1]|\right)$ and $\pp$ be a prime ideal from $P$. Then
    \begin{equation}
        \begin{split}
        eU_s \overset{(\ref{lem_infinEdge_e2})}{\equiv}
        \sum_{i=1}^{e}U_{p^is} \equiv
        \sum_{i=1}^{e}  \sum_{j=0}^{d-1} l_j\alpha_{j}^{p^is} \equiv
        \sum_{j=0}^{d-1} l_j\sum_{i=1}^{e}  \alpha_{j}^{p^is} \overset{(P3)}{\equiv}
        \sum_{j=0}^{d-1} l_j \frac{e}{|[j]|}\sum_{k\in [j]}  \alpha_{k}^{s} \equiv\\
        \equiv \sum_{0\leqslant j,k<d,\, k\sim_\pp j}l_j\alpha_{k}^{s} \frac{e}{|[j]|}\equiv
        \sum_{k=0}^{d-1} \alpha_k^s \frac{e}{|[k]|}\sum_{j\in [k]}  l_j \equiv
        e\sum_{k=0}^{d-1} \alpha_k^s \frac{\sum_{j\in [k]}l_j }{|[k]|}  
        \pmod \pp .
    \end{split}
    \end{equation}

    Let us define a sequence ${\mathbf V}$ as follows $V_n := \sum_{j=0}^{d-1} \alpha_j^n \frac{\sum_{k\in [j]}l_k }{|[j]|}$. The above formula states that 
    \begin{align}
        eU_s \equiv eV_s \pmod \pp \quad\quad \text{for every $s$ and $\pp\in P$}
    \end{align}
    The sequences ${\mathbf U}$ and ${\mathbf V}$ must be thus equal at every point. Analogously to the proof of Theorem \ref{theor_Un} it can be shown that the coefficients of sequences ${\mathbf U}$ and ${\mathbf V}$ are equal, in particular the coefficients of $\alpha_s$ and $\alpha_t$ are:
    \begin{align}
        l_s = \frac{\sum_{k\in [s]}l_k }{|[s]|} \qquad\text{and}\qquad   l_t = \frac{\sum_{k\in [t]}l_k }{|[t]|}.
    \end{align}
    Since $[s]=[t]$ the coefficients $l_s$, $l_t$ are equal.

\end{proof}

\begin{lemma}\label{lem_irred}
    If $\alpha_i$ and $\alpha_j$ are roots of the same irreducible factor of $C_{\mathbf U}$, then there are infinitely many edges between $i$ and $j$ in the graph $G_\mathbf{U}$.
\end{lemma}
\begin{proof}
    

    Let us remove finitely many edges from the graph $G_
    \mathbf{U}$.
    For every removed edge we add its label to the set $B$. We also add $\pp$ to the set $B$ if $\pp|disc(K)$. Therefore $B$ is a finite subset of prime ideals. Form now on we will consider only prime ideals that are in $P\setminus B$ where $P$ is the set of all prime ideals of $\OO_K$. 
    
    Let $D\subseteq \{0,\ldots,d-1\}$ be a connected component of the modified graph. We will show that $D$ must have a certain form.
    
    Notice that the Frobenius automorphism $\Phi_\pp$ for $\pp\in P\setminus B$ permutes the roots $\alphas$ such that:
    \begin{align}\label{lem_irred_e1}
        \text{for any index } \ s\in D\ \text{ if } \Phi_\pp(\alpha_s) \equiv \alpha_t \pmod \pp \ \text{ for some}\ t \text{, then } \ t\in D
    \end{align}

    Let $e_k(x_0,\ldots,x_t)=\sum_{1\leq i_1<\ldots <i_k\leq t}\alpha_{i_1},\cdots \alpha_{i_k}$ be the $k$-th elementary symmetric polynomial. 
    With $E_k$ we will denote
    \begin{align}\label{e_k_def}
        E_k := e_k(\alpha_{i_1}, \ldots, \alpha_{i_t}) \qquad\text{for $\{i_1,\ldots,i_t\}=D$}
    \end{align}
    By (\ref{lem_irred_e1}) we have
    \begin{align} \label{lem_irred_Eq_fix}
        \Phi_\pp(E_k) \equiv E_k \pmod \pp \qquad\text{for any $\pp\in P\setminus B$}.
    \end{align}
    Since $\Phi_\pp$ is a generator of the Galois group $H:=\gal\left((\OO_K/\pp)/ (\Z/p)\right)$, every $\sigma\in H$ must have the form $\sigma=\Phi_\pp^k$ for some $k$. If some $x\in \OO_K/\pp$ is a fixed point of $\Phi_\pp$, then it must be a fixed point of every automorphism in $H$ which implies that $x$ is in fact in $\Z/p$.  

    This is exactly the case for $E_k$. By (\ref{lem_irred_Eq_fix}) it is a fixed point of the Frobenius automorphism. Therefore, the reduction of $E_k$ modulo $\pp$ must be in $\Z/p$.
    Hence, there exists $n_{k,\pp}\in\Z$ such that $E_k \equiv n_{k,\pp} \pmod \pp $. Let $W_k(x)$ be the minimal polynomial of $E_k$. We can see that $W_k(n_{k,\pp})\equiv 0\pmod \pp$ for any $\pp\in P\setminus B$. Thus the polynomial $W_k(x)$ has a root modulo $\pp$ for almost every $\pp\in P$. This means that polynomial $W_k(x)$ must be linear (see Appendix). 
    By definition $E_k$ is a root of $W_k(x)$, so $E_k \in \Q$ and by (\ref{e_k_def}) $E_k$ can be written as a polynomial of algebraic integers with integer coefficients, so $E_k\in \Z$.

    Let $V_D(x)$ be a polynomial defined as
    \begin{align}
        V_D(x) := \prod_{k\in D} (x-\alpha_k) = \sum_{k=0}^{|D|}(-1)^k E_kx^{|D|-k}.
    \end{align}
    $V_D(x)$ has integer coefficients and by definition it is a non-costant factor of $C_{\mathbf U}$ - the characteristic polynomial of $\textbf{U}$. We can decompose polynomial $C_{\mathbf U}$ into irreducible, pairwise different factors: $C_{\mathbf U} = C_1C_2\cdots C_m$. $V_D$ must be a product of at least one of those irreducible factors and if $\alpha_i$ is a root of one of them, then it is a root of $V_D$ and $i\in D$.
    
    Let $D_1, D_2, \ldots, D_w$ be connected components of modified graph and let $\alpha_s$ and $\alpha_t$ be roots of an irreducible component $C_i$. Obviously $s$ must be a member of a component $D_j$ for some $j$. Then $C_i|V_{D_j}$, so every root of $C_i$ is in $D_j$, in particular $\alpha_t$.

    The choice of the finite set $B$ was arbitrary. No matter how many edges we remove form the graph $G_\mathbf{U}$. As long as we remove finite amount of them, then all roots of an irreducible factor will fall into the same connected component. In other words: if $\alpha_s$ and $\alpha_t$ are roots of the same irreducible factor, then there are infinitely many edges between $s$ and $t$ in the graph $G_\mathbf{U}$.

\end{proof}







We can finally apply the two above lemmas to the case of $\Delta_\mathbf{U} \not=0$.

\begin{theorem}\label{theor_general_form}
    If $\Delta_{\mathbf U}\not=0$, then ${\mathbf U}$ almost satisfies the Dold condition if and only if ${\mathbf U}$ has the form 
    \begin{align}\label{eq_general_form}
        \mathbf U_n = l_1\cdot\sum_{C_1(\beta)=0} \beta^n   \ \ +\ \cdots\ + \ \
        l_m\cdot\sum_{C_m(\beta)=0} \beta^n ,   
    \end{align}
    where $C_1, C_2, \ldots, C_m$ are irreducible factors of $C_{\mathbf U}$ and $l_1, l_2, \cdots, l_m \in K$.
\end{theorem}

\begin{proof}

$(\Rightarrow)$. There exists a non-zero constant $N\in\Z$ such that $N\mathbf{U}$ satisfies the Dold condition. Let $\alpha_s$ and $\alpha_t$ be two roots of the same irreducible factor of $C_{\mathbf{U}}$. By Lemma \ref{lem_irred} there are infinitely many edges between vertices $s$ and $t$ in the graph $G_{(N\mathbf{U})}$ and by Lemma \ref{lem_infinEdge} the coefficients $Nl_s$ and $Nl_t$ must be equal. This clearly implies that the sequence $\mathbf{U}$ must have the form (\ref{eq_general_form}).

$(\Leftarrow)$  We have already seen (corollary \ref{c_sum}) that $  V^{(i)}_n := \sum_{C_i(\beta)=0} \beta^n$ satisfies the Dold condition. We also now that if some sequences satisfy the Dold condition, then their sum does as well. The result follows easily.
        
\end{proof}

With the above theorem we have proved the implication "$\Rightarrow$" in Theorem \ref{the_part}.

\begin{corollary}[the second part of the proof of Theorem \ref{the_part} ] 

If a sequence $\mathbf{U}$ almost satisfies the Dold condition then $l_0=\cdots=l_{d-1}$.
\end{corollary}


All our previous bounds on the $\fail$ factor required that $\deg C_\mathbf{U} = 2$ or $C_\mathbf{U}$ is irreducible. None of these assumption apply in this case. We are forced to derive another estimation, this time unfortunately much weaker.

\begin{lemma}\label{lem_fail_del}
    If $\Delta_\mathbf{U} \not= 0$ and the sequence $\mathbf{U}$  almost satisfies the Dold condition, then $\fail(\mathbf{U}) | r_d\Delta_\mathbf{U}$.
\end{lemma}
\begin{proof}

We consider the equation

\begin{align}
    \begin{bmatrix}
        \alpha_0^{1} & \alpha_1^{1} & \cdots & \alpha_{d-1}^{1}\\
        \alpha_0^{2} & \alpha_1^{2} & \cdots & \alpha_{d-1}^{2}\\
        \vdots & \vdots & \ddots & \vdots \\
        \alpha_0^{d} & \alpha_1^{d} & \cdots & \alpha_{d-1}^{d}
    \end{bmatrix}
    \cdot
    \begin{bmatrix}
        l_0 \\ l_1 \\ \vdots \\ l_{d-1}
    \end{bmatrix}
    =
    \begin{bmatrix}
        U_1 \\  U_2 \\ \vdots \\  U_{d}
    \end{bmatrix}
\end{align}

The square matrix is invertible because it is a Vandermonde matrix with different columns, where $i$-th column is multiplied with $\alpha_i$. Let us denote this matrix by $M$. Let $\adj M$ be the adjugate of $M$. Then  $M\adj M  = \det M \cdot Id$ where $Id$ is the identity matrix of appropriate order. The following is true

\begin{align}\label{eq_li}
    \begin{bmatrix}
        l_0 \\ l_1 \\ \vdots \\ l_{d-1}
    \end{bmatrix}
    =
    \frac{1}{\det M}
    \cdot
    \adj M
    \cdot
    \begin{bmatrix}
         U_1 \\  U_2 \\ \vdots \\  U_{d}
    \end{bmatrix}
\end{align}

The numbers $ U_1, \ldots  U_{d}$ are all integers and the entries of the matrix $\adj M$ are polynomials of $\alphas$ with integer coefficients. From (\ref{eq_li}) we conclude that the numbers $l_i$ are of the form $l_i = \frac{1}{\det M}W_i(\alphas)$ for some polynomials $W_i\in\Z[X_0,\ldots ,X_{d-1}]$.

By definition, $\Delta_{\mathbf U} = \prod_{0\leqslant i<j<d}(\alpha_i-\alpha_j)^2$. Hence $\Delta_{\mathbf U}$ is a symmetric polynomial in $\alpha_0,\ldots,\alpha_{d-1}$, so the value of $\Delta_{\mathbf U}$ is an integer. From the formula for the determinant of the Vandermonde matrix we see that $\det M = \pm \alpha_0\cdots\alpha_{d-1} \sqrt{\Delta_{\mathbf U}} =  \pm r_d \sqrt{\Delta_{\mathbf U}}$.

The coefficients $l_i$ must therefore satisfy the following
\begin{align} \label{eq_li2}
    l_i = \frac{1}{r_d\sqrt{\Delta_\mathbf{U}}}W_i(\alphas)
\end{align}

If a sequence $\mathbf{U}$ almost satisfies the Dold condition, then by Theorem \ref{theor_general_form} it has the form 

\begin{align}
        U_n = l_1\cdot\sum_{C_1(\beta)=0} \beta^n   \ \ +\ \cdots\ + \ \
        l_m\cdot\sum_{C_m(\beta)=0} \beta^n ,   
\end{align}
where $C_1, C_2, \ldots, C_m$ are irreducible factors of $C_{\mathbf U}$ and $l_1, l_2, \cdots, l_m \in K$. By equation (\ref{eq_li2}) the denominators of $l_i$ must divide $\Delta_\mathbf{U}r_d$ so $l_i\Delta_\mathbf{U}r_d \in \OO_K$. By Corollary \ref{c_sum} the sequence ${W}^{(i)}_n := \sum_{C_i(\beta)=0} \beta^n$ satisfies the Dold condition for every $i$. The sequence $\mathbf{W}^{(i)}$ multiplied by an algebraic integer also satisfies the Dold condition and sum of the sequences that satisfy it also satisfies it. With that in mind it is easy to see that 

\begin{align}
        \Delta_\mathbf{U}r_d U_n = \Delta_\mathbf{U}r_dl_1\cdot\sum_{C_1(\beta)=0} \beta^n   \ \ +\ \cdots\ + \ \
        \Delta_\mathbf{U}r_dl_m\cdot\sum_{C_m(\beta)=0} \beta^n ,   
\end{align}
so $\Delta_\mathbf{U}r_d\mathbf U$ satisfies the Dold condition.

\end{proof}

Looking at the above theorem and the estimation of the $\fail$ factor in the irreducible case one might think that the estimation might be improved by just adding $r_d\Delta_\mathbf{U}$ term to the greatest common divisor formula. But this in fact changes nothing because $\Delta_\mathbf{U}$ can be expressed as linear combination of $ir_i$ for $1\leqslant i\leqslant d$. 

$\Delta_\mathbf{U}$ is the discriminant of the characteristic polynomial $C_\mathbf{U}$. We already used Vandermonde matrix to calculate this value but there is also another formula (which is in fact the definition) for the discriminant, namely $\Delta_\mathbf{U} = \res(C_\mathbf{U}, C'_\mathbf{U})$ where $\res$ is the resultant of polynomials. This can be expressed as a determinant of a certain $2d-1\times 2d-1$ matrix. For example, for $d=4$ the mentioned matrix for $C_\mathbf{U}(x) = x^4-r_1x^3-r_2x^2-r_3x -r_4$ is the following.

\begin{align}
    \begin{bmatrix}
        -r_4    &0      &0         &-r_3   &0      &0      &0      \\
        -r_3    &-r_4   &0         &-2r_2  &-r_3   &0      &0      \\
        -r_2    &-r_3   &-r_4      &-3r_1  &-2r_2  &-r_3   &0      \\
        -r_1    &-r_2   &-r_3      &4      &-3r_1  &-2r_2  &-r_3   \\
        1       &-r_1   &-r_2      &0      &4      &-3r_1  &-2r_2  \\
        0       &1      &-r_1      &0      &0      &4      &-3r_1  \\
        0       &0      &1         &0      &0      &0      &4      
    \end{bmatrix}.
\end{align}

If we multiply the third (in general $d-1$-st) column by 4 (in general $d$) and subtract the result from the last column the determinant will not change but the matrix will have the following form

\begin{align}
    \begin{bmatrix}
            &-r_4   &0      &0        &-r_3   &0      &0      &0\\
            &-r_3   &-r_4   &0        &-2r_2  &-r_3   &0      &0\\
            &-r_2   &-r_3   &-r_4     &-3r_1  &-2r_2  &-r_3   &4r_4\\
            &-r_1   &-r_2   &-r_3     &4      &-3r_1  &-2r_2  &3r_3\\
            &1      &-r_1   &-r_2     &0      &4      &-3r_1  &2r_2\\
            &0      &1      &-r_1     &0      &0      &4      &r_1\\
            &0      &0      &1        &0      &0      &0      &0
    \end{bmatrix}.
\end{align}

Now we can see that each nonzero element of the last column is of the form $ir_i$ so the determinant of this matrix (which is the discriminant of $C_\mathbf{U}$) is a linear combination of these terms. Hence,
\begin{align}
    \gcd\left(r_1,2r_2, \ldots , dr_d, \Delta_\mathbf{U}\right) 
    \quad = \quad   
    \gcd\left(r_1,2r_2, \ldots, dr_d\right)
\end{align}

\subsection{Arbitrary Sequences}

Up until this point we almost always considered sequences $\mathbf{U}$ with $\Delta_{\mathbf{U}} \not= 0$. It simplifies calculations and the exact formula for the sequence  $\mathbf{U}$. With all the lemmas already established it is not hard to classify which sequences (with $\Delta = 0$) satisfy the Dold condition. 

In general, if a sequence $\mathbf{U}$ is defined by linear recursion then $U_n$ can be written as
\begin{align}
    U_n = \sum_{i=0}^{d-1} L_i(n)\alpha_i^n,
\end{align}
where for each $i\in\{0,\ldots ,d-1\}$, $L_i(x) \in K[x]$ is a polynomial of degree less than the multiplicity of $\alpha_i$ as a root of the polynomial $C_{\mathbf{U}}$.

\begin{lemma}\label{lem_poly_zero}
    If for some polynomials $L_0(x), \ldots, L_{d-1}(x) \in \Z[x]$ we have\linebreak ${V_n} := \sum_{i=0}^{d-1}L_i(n)\alpha_i^n = 0$ for every $n$ then the polynomials $L_i(x)$ are all zero. 
\end{lemma}

\begin{proof}
    This proof consists of two parts. In the first part we will show that with the conditions as in the statement of the lemma the constant terms of  the polynomials $L_i$ are zero. In the second part we will use the former one to prove the lemma.

    Let $p$ be any prime not ramified in the splitting field of $C_{\mathbf{U}}$ and not dividing denominators of coefficients of $L_i$. There are infinitely many such primes.
    Consider the following matrices

    \begin{align}
        M_p :=
        \begin{bmatrix}
            \alpha_0^{0} & \alpha_1^{0} & \cdots & \alpha_{d-1}^{0}\\
            \alpha_0^{p} & \alpha_1^{p} & \cdots & \alpha_{d-1}^{p}\\
            \vdots & \vdots & \ddots & \vdots \\
            \alpha_0^{(d-1)p} & \alpha_1^{(d-1)p} & \cdots & \alpha_{d-1}^{(d-1)p}
        \end{bmatrix}
        ,\quad
        M :=
        \begin{bmatrix}
            \alpha_0^{0} & \alpha_1^{0} & \cdots & \alpha_{d-1}^{0}\\
            \alpha_0^{1} & \alpha_1^{1} & \cdots & \alpha_{d-1}^{1}\\
            \vdots & \vdots & \ddots & \vdots \\
            \alpha_0^{d-1} & \alpha_1^{d-1} & \cdots & \alpha_{d-1}^{d-1}
        \end{bmatrix}
    \end{align}

    Since $p$ is not ramified, raising $\alphas$ to the power of $p$ permutes them $\pmod\pp$. One can see that viewing the columns of the matrix $M_p$ modulo $\pp$ they are permuted columns of the matrix $M$ so $\det M_p \equiv \pm\det M \pmod \pp$. If we now add a requirement to the prime $p$ that $p\not|\det M$ we get $\det M_p \not\equiv 0 \pmod \pp$.

    We can see that for every $s\in \N$ we have
    $$
        0 \equiv
        {V}_{ps} \equiv 
        \sum_{i=0}^{d-1}L_i(ps)\alpha_i^{ps} \equiv 
        \sum_{i=0}^{d-1}L_i(0)\alpha_i^{ps} =: {V}_{ps}'
        \pmod{\pp}
    $$

    Consider the congruence
    
    \begin{align}
        M_p
        \cdot
         \begin{bmatrix}
           L_0(0) \\ L_1(0) \\ \vdots \\ L_{d-1}(0)
        \end{bmatrix}
        \equiv
          \begin{bmatrix}
           {V}_{0}' \\ {V}_{p}' \\ \vdots \\ {V}_{(d-1)p}' 
        \end{bmatrix}
         \equiv
          \begin{bmatrix}
           0 \\ 0 \\ \vdots \\ 0
        \end{bmatrix}
        \pmod \pp.
    \end{align}

   The determinant of $M_p$ is non-zero modulo $\pp$. If then the multiplication of the matrix with a vector yields the zero vector, then the mentioned vector must also be zero. Hence $L_i(0) \equiv 0\pmod\pp$ for every $i$. This equivalence is true for infinitely many prime ideals $\pp$. Therefore $L_i(0) = 0$ for every $i$.

    This implies that the constant term of every polynomial $L_i$ is zero, which concludes the first part of the proof.

    For the second part we will prove that if the last $k\in\N$ terms of each polynomial $L_i$ are zero, then in fact the last $k+1$ terms are zero.  To prove that we take polynomials $L_i'(x) := \frac{L_i(x)}{x^k}$ and observe that the sequence $\mathbf{W}$ defined as 
    $$ {W}_n := \frac{{V}_n}{n^k} =  \sum_{i=0}^{d-1}L_i'(n)\alpha_i^n $$
    is also zero at each point. The sequence ${W}$ therefore satisfies the assumption of the lemma and therefore the constant terms of the polynomials $L_i'$ are zero. This immediately implies that the last $k+1$ terms of polynomials $L_i(x) = x^kL'_i(x)$ are also zero.

    By induction it can be easily observed that polynomials $L_i$ must all be zero.    
    


\end{proof}

\begin{lemma}\label{lem_infinEdge2}
    If a sequence ${\mathbf U}$ satisfies the Dold condition and there exist infinitely many edges between vertices $s$ and $t$, then $L_s=L_t$ and $\deg L_s = \deg L_t = 0$.
\end{lemma}
\begin{proof}
    We can define the set $P$ in the same way we did it in Lemma \ref{lem_infinEdge} such that it satisfies the conditions (P1), (P2) and (P3). 

    Let $\mathbf{U}'$ be the sequence defined as 
    \begin{align}
        U'_n := \sum_{i=0}^{d-1} L_i(0)\alpha^n_i .
    \end{align}
    Let $N$ be a nonzero number such that $NL_i(x)\in \OO_K[x]$ for any $0\leqslant i <d$. We can see that
    
    \begin{align}\label{lem_infinEdge2_e1}
        eN  U_s \equiv
        \sum_{i=1}^e N  U_{p^is} \equiv
        \sum_{i=1}^e N  U'_{p^is} \equiv \cdots \equiv
        e\sum_{k=0}^{d-1} \frac{\sum_{j\in [k]}NL_j(0) }{|[k]|}  \alpha_k^s
        \pmod \pp.
    \end{align}
    The elided formulae are analogous to (\ref{lem_infinEdge_e2}). 
    We define ${V}_n := \sum_{k=0}^{d-1} \frac{\sum_{j\in [k]}L_j(0) }{|[k]|}  \alpha_k^s$ and see that since (\ref{lem_infinEdge2_e1}) is valid for infinitely many primes $\pp$ then $eN U_s = eN{V}_s$ for every $s$ which implies $\mathbf{U} = \mathbf{V}$.

    The difference $\mathbf{U} - \mathbf{V}$ is therefore a sequence that is always equal to zero. By the first part of this Lemma the coefficients of  $\mathbf{U} - \mathbf{V}$ are all zero polynomials, which implies that 
    \begin{align}
        L_k = \frac{\sum_{j\in [k]}NL_j(0) }{|[k]|}
        \quad\quad \text{for every $0\leqslant k < d$}.
    \end{align}
    The degree of the right hand side polynomial is 0, so the degree of the left hand side must be also 0, hence $\deg L_i = 0$ for every $0\leqslant i  < d$. 
    If there are infinitely many edges between vertices $s$ and $t$ then $[s] = [t]$ and $ L_s = \frac{\sum_{j\in [s]}NL_j(0) }{|[s]|} = \frac{\sum_{j\in [t]}NL_j(0) }{|[t]|} = L_t$.

\end{proof}

With this Lemma and Corollary \ref{c_sum} it is clear the following.
\begin{theorem}\label{the_any}
    The sequence ${\mathbf U}$ almost satisfies the Dold condition if and only if ${\mathbf U}$ has the form 
    \begin{align}\label{eq_general_form_2}
        U_n = l_1\cdot\sum_{C_1(\beta)=0} \beta^n   \ \ +\ \cdots\ + \ \
        l_m\cdot\sum_{C_m(\beta)=0} \beta^n ,   
    \end{align}
    where $C_1, C_2, \ldots, C_m$ are irreducible factors of $C_{\mathbf U}$ and $l_1, l_2, \cdots, l_m \in K$
\end{theorem}


The condition $\Delta_U=0$ implies that the characteristic polynomial has at least one multiple root. We can consider a polynomial $W := \rad(C_\mathbf{U}) = x^{d'} - \sum_{i=0}^{d'-1}x^is_{d'-i}$. This polynomial has degree $d'< d$ and its coefficients also form a recurrence formula for the sequence $\mathbf{U}$ 
$$ {U}_n = s_{1}{U}_{n-1} + \cdots + s_{d'}{U}_{d'}.
$$

We can now forget that $C_\mathbf{U}$ was the characteristic polynomial and pretend $W$ took this role. As $\disc(W) \not= 0$ we can use Lemma \ref{lem_fail_del} to see that $s_d'\disc(W)$ is a multiple of $\fail(\mathbf{U})$. We can also see that since $W|C_\mathbf{U}$ then $s_{d'}|r_d$ so
$$\fail(\mathbf{U})\ |\ r_d\disc(\rad(C_\mathbf{U})). $$

\section{The Dold condition for sequences $({U}_{n^t})_{n\in\N_+}$}

The following result is similar to Theorem \ref{theor_n2_final}. To be more precise, we focus on the validity of the Dold condition for the subsequence of an arbitrary linear recurrent sequence $\textbf{U}$ sampled over the powers of positive integers of a fixed exponent $t$. We show that if $t$ is a multiple of the degree of the splitting field of $C_{\textbf{U}}$ over $\Q$, then the sequence $(\rad(\Delta_K) {U}_{n^t})_{n\in\N_+}$ satisfies the Dold condition.

\begin{theorem}\label{the_n_d}
    Let $\textbf{U}$ be linear recurrent sequence with characteristic polynomial $C_{\textbf{U}}=\prod_{i=0}^{d-1}(x-\alpha_i)$ and $\Delta_\mathbf{U}\not=0$. Let $K$ be the splitting field of $C_{\textbf{U}}$ and $m=[K:\Q]$. The sequence $(r_d\Delta_\mathbf{U}\rad(\Delta_K) {U}_{n^t})_{n\in\N_+}$ satisfies the Dold condition for any $t\in\N_+$ such that $m|t$.
\end{theorem}
\begin{proof}
    Let $U_n = \sum_{i=0}^{d-1} l_i\alpha_i^n$
    In the proof of the Lemma \ref{lem_fail_del} we saw that $r_d\Delta_\mathbf{U}l_i$ is always an algebraic integer for any $0\leqslant i<d$. 

    If we focus on showing that $(r_d\Delta_\mathbf{U}\rad(\Delta_K) \mathbf{U}_{n^m})$ satisfies the Dold condition, the rest will follow by the similar argument to the one we used in the proof of Theorem \ref{theor_n2_final}.
    It will be enough to show that 
    \begin{align}\label{the_n_d_e1}
        \rad(\Delta_K)\left(x^{p^{km}} - x^{p^{(k-1)m}}\right) 
        \equiv 0 \pmod {\pp^{ek}} 
        \quad\quad \text{for any $x\in \OO_K$}
    \end{align}
    because then 
    \begin{align}
        &r_d\Delta_\mathbf{U}\rad(\Delta_K) {U}_{(sp^k)^m} -  r_d\Delta_\mathbf{U}\rad(\Delta_K) {U}_{(sp^{k-1})^m} 
        \\ &\equiv \rad(\Delta_K)\sum_{i=0}^d (r_d\Delta_\mathbf{U}l_i)\alpha_i^{(sp^k)^m} - \rad(\Delta_K)\sum_{i=0}^d (r_d\Delta_\mathbf{U}l_i)\alpha_i^{(sp^{k-1})^m}
        \\&\equiv  \sum_{i=0}^d(r_d\Delta_\mathbf{U} l_i)\rad(\Delta_K)\left(\left(\alpha_i^{s^m}\right)^{p^{km}} - \left(\alpha_i^{s^m}\right)^{p^{(k-1)m}} \right)
        \\  &\overset{(\ref{the_n_d_e1})}{\equiv} 0 \pmod {\pp^{ek}} \quad\quad\text{for any $k,s\in\N_+$}
    \end{align}

which shows that the sequence $(r_d\Delta_\mathbf{U}\rad(\Delta_K) \mathbf{U}_{n^m})$ satisfies the Dold condition.

To prove (\ref{the_n_d_e1}) we consider two cases.
\begin{enumerate}
    \item Prime ideal $\pp$ is not ramified and $e = 1$. Let $f$ be the inertial degree associated with the prime ideal $\pp$. We know that $ef|m$. The field $\OO_K/\pp$ has $p^f$ elements so $x^{p^f}\equiv x \pmod \pp$ for every $x\in\OO_K$. Because $f|m$ we have $x^{p^m}\equiv x \pmod \pp$. Now by Lemma \ref{lem_allInOne} we have $x^{p^{m+h}}\equiv x^{h} \pmod {\pp^{h+1}}$ for any $h\in\N_+$. Substituting $h = d(k-1)$ we get $x^{p^{km}}\equiv x^{p^{(k-1)m}} \pmod {\pp^{m(k-1)+1}}$ and this implies (\ref{the_n_d_e1}) because $m(k-1)+1 \geqslant k$.

    \item  The ideal $\pp$ is ramified and $\pp|\Delta_K$. Again we have $x^{p^{m}} \equiv x \pmod \pp$ for any $x\in\OO_K$. By Lemma \ref{lem_allInOne}, raising both sides to the power of ${p^{e-1}}$ we get $x^{p^{m+e-1}} \equiv x^{p^{e-1}} \pmod {\pp^e}$. Raising both sides to the power of $p^{m-e+1}$ gives $x^{p^{2m}} \equiv x^{p^{m}} \pmod {\pp^e}$. By Lemma \ref{lem_allInOne} we have $x^{p^{2m+h}} \equiv x^{p^{m+h}} \pmod {\pp^{e+he}}$. Substituting $h=m(k-2)$ (notice $k\geqslant 2$) yields $x^{p^{mk}} \equiv x^{p^{m(k-1)}} \pmod {\pp^{e(1+m(k-2))}}$. Now we multiply by $p$ and get   $px^{p^{mk}} \equiv px^{p^{m(k-1)}} \pmod {\pp^{e(2+m(k-2))}}$. Because $2+m(k-2) \geqslant k$ this finally implies (\ref{the_n_d_e1}).
\end{enumerate}

\end{proof}

\begin{example}
   Using the same recurrence formula as previously $U_{n} = 10U_{n-2} - U_{n-4}$ for $n>4$ but with $U_1 = 1$, $U_2 = 0$, $U_3=9$, $U_4=0$ yields the formula
   $$
    {U}_n = 
        \frac{\sqrt{3}}{12}\left(\sqrt{2}+\sqrt{3}\right)^n
        -\frac{\sqrt{3}}{12}\left(\sqrt{2}-\sqrt{3}\right)^n
        +\frac{\sqrt{3}}{12}\left(-\sqrt{2}+\sqrt{3}\right)^n
        -\frac{\sqrt{3}}{12}\left(-\sqrt{2}-\sqrt{3}\right)^n
    $$
    One can see that the sequence $\textbf{U}$ is not realizable as the coefficients near the powers of the roots of $C_{\textbf{U}}$ differ. But according to the above, the sequence $({U}_{n^4})_{n\in\N_+}$ almost satisfies the Dold condition  and its repairing factor must divide $\rad\left(\Delta_{\Q(\sqrt{2}, \sqrt{3})}\right) = \rad(2^{10}\cdot3^2) = 6$. On the other hand, since $2\nmid U_{2^4}-U_{1^4}$ and $3\nmid U_{3^4}-U_{1^4}$, we have that $6$ is the repairing factor of $({U}_{n^4})_{n\in\N_+}$.
    
\end{example}

\section{Everything, all at once}

The following table shows all the results of this paper in one place. The first column contains all the considered cases, the second the  necessary and sufficient condition for the sequence to almost satisfy the Dold condition and the third column shows a bound on the $\fail$ term, a value that $\fail(\mathbf{U})$ must divide.

\begin{center}
\begin{tabular}{||c || c | c | c ||} 
\hline
    Restrictions on $C_\mathbf{U}$ & 
    \makecell{Condition for alm. \\sat. Dold cond.} & 
    Known multiple of $\fail$ & 
    References  \\ 
\hline\hline
    $d=1$ & 
    always & 
    $r_1 = \alpha_1$ & 
    Tr. \\
\hline
    $d=2$, irreducible & 
    $l_1=l_2$ & 
    $\gcd\left(r_1,2r_2\right)$ & 
    Th.\ref{TheoremNotSquare}, Lem.\ref{lem_gcd}\\
\hline
    $d=2$, reducible & 
    $\deg L_1=0 $ & 
    \makecell{$r_2\rad(\Delta_\mathbf{U})$ \\ (only when $\Delta_\mathbf{U}\not=0$)} & 
    Th. \ref{theor_rad} \\
\hline
    \nice & 
    $l_0=\cdots=l_{d-1}$ & 
    $\gcd\left(r_1,2r_2, \ldots, dr_d \right)$ & 
    Th.\ref{the_nice}, Lem.\ref{lem_gcd}\\
\hline
    irreducible & 
    $l_0=\cdots=l_{d-1}$ & 
    $\gcd\left(r_1,2r_2, \ldots, dr_d \right)$  & 
    Th.\ref{the_part}, Lem.\ref{lem_gcd}\\
\hline
    $\Delta_\mathbf{U}\not=0$  & 
    \makecell{Coefficients near\\ the powers of the roots\\ of the same  irreducible \\factor are equal}  &
    $r_d\Delta_\mathbf{U} $ & 
    Th.\ref{theor_general_form}, Lem.\ref{lem_fail_del}\\
\hline
    any  & 
    $\deg L_i = 0$ and  the above & 
    $r_d\disc(\rad(C_\mathbf{U}))$ & 
    Lem.\ref{the_any}\\
\hline\hline
    \makecell{Sequence $(\mathbf{U}_{n^{km}})$ \\ $\Delta_\mathbf{U} \not= 0$}  & 
    always & 
    $r_d\Delta_\mathbf{U}\rad(\Delta_K)$ & 
    Th.\ref{the_n_d}\\
\hline\hline
\end{tabular}


\end{center}

\section{Acknowledgments}

I would like to express my gratitude to my advisor, Piotr Miska, for suggesting this topic, fixing my mistakes and endless pieces of advice.

\appendix
\section{}

\begin{theorem}
    If an irreducible, non-constant polynomial $W(x)\in\Z[x]$ has a root modulo $p$ for almost every prime $p$, then $W(x)$ is a linear polynomial.
\end{theorem}

\begin{proof}
    Let $\alphas$ be the roots of this polynomial. Again, we consider a field $K := \Q(\alphas)$ and the ring of algebraic integers $\OO_K$. The extension $K/\Q$ is Galois with discriminant $\Delta$.

    Every element of $G := Gal(K/\Q)$ permutes the roots of $W$.
    By the Chebotarev Density Theorem, if $C\subset Gal(K/\Q)$ is a conjugacy class in $G$, then 
    $$\{\pp: \pp \text{ prime ideal, } \pp \not| \Delta, \Phi_\pp \in C\}$$
    has density $\frac{|C|}{|G|}$.




    Let $\pp\not|\Delta$ be a prime ideal such that the polynomial $W$ has a root modulo $\pp$. Given an irreducible (over $\Z/p$) factor of $W$, the Frobenius automorphism $\Phi_\pp$ permutes  its roots  transitively. Because $W$ has a root modulo $\pp$, the permutation associated with $\Phi_\pp$ has a fixed point.

    For almost all $\pp$ the $\Phi_\pp$ has a fixed point, so by the Chebotarev Density Theorem all the elements of the Galois group $G$ have a fixed point. 

    As assumed, the polynomial $W$ is irreducible, which implies that the group $G$ act transitively on the set of roots of $W$. An elementary result of group theory says that if a finite group acts transitively and each element of this group has a fixed point, then this group is trivial. 

    In this context this means that the Galois group $G$ contains only the identity, which implies that the polynomial $W$ has degree one.

\end{proof}

\end{document}